\documentclass[leqno,12pt]{article}
\usepackage{amsfonts}
\usepackage{amsmath, amssymb}
\usepackage{amsthm}

\theoremstyle{plain}
\newtheorem{theorem}{\indent\sc Theorem}[section]
\newtheorem{lemma}[theorem]{\indent\sc Lemma}
\newtheorem{corollary}[theorem]{\indent\sc Corollary}
\newtheorem{proposition}[theorem]{\indent\sc Proposition}

\theoremstyle{definition}
\newtheorem{definition}[theorem]{\indent\sc Definition}
\newtheorem{remark}[theorem]{\indent\sc Remark}
\newtheorem{example}[theorem]{\indent\sc Example}

\newcommand\on{\operatorname}
\renewcommand\div{\on{div}}

\newcommand\Ric{\on{Ric}}
\newcommand\scal{\on{scal}}

\newcommand\trace{\on{trace}}


\begin{document}

\title{Conformal-projective transformations on statistical and semi-Weyl manifolds with torsion}
\author{Adara M. Blaga and Antonella Nannicini}
\date{}
\maketitle

\begin{abstract}
We show that statistical and semi-Weyl structures with torsion are invariant under conformal-projective transformations.
We prove that a non-degenerate submanifold of a semi-Weyl (respectively, statistical) manifold with torsion is also a semi-Weyl (respectively, statistical) manifold with torsion, and that the induced structures of two conformal-projective equivalent semi-Weyl (respectively, statistical) structures with torsion on a manifold to a non-degenerate submanifold, are conformal-projective equivalent, too. Also, we prove that the
umbilical points of a non-degenerate hypersurface in a semi-Weyl manifold with torsion are preserved by conformal-projective changes. Then we consider lightlike hypersurfaces of semi-Weyl manifolds with torsion and we describe similarities and differences with respect to the non-degenerate hypersurfaces. Finally, we show that a semi-Weyl manifold with torsion can be realized by a non-degenerate affine distribution.
\end{abstract}

\markboth{{\small\it {\hspace{0.5cm} Conformal-projective transformations on statistical and semi-Weyl manifolds with torsion}}}{\small\it{Conformal-projective transformations on statistical and semi-Weyl manifolds with torsion \hspace{0.5cm}}}


\bigskip

\section{Introduction}

Statistical structures play a central role in information geometry. They are defined as pairs of a torsion-free affine connection and a pseudo-Riemannian metric satisfying the Codazzi equation. Later on, this notion was extended to statistical structures with torsion.
Generalizing both these concepts, semi-Weyl structures with torsion, introduced in \cite{blna}, make a deeper connection to affine differential geometry.
Moreover, conformal and projective transformations are through the most important transformations in the theory of affine connections, crucial in conformal and projective geometry.

Having these in mind, the aim of the present paper is to present some properties of semi-Weyl structures with torsion, with a special view towards these kind of transformations.
Precisely, we prove their invariance under conformal-projective transformations and that the induced structure of two conformal-projective equivalent semi-Weyl structures with torsion on a pseudo-Riemannian manifold to a pseudo-Riemannian submanifold, are conformal-projective equivalent, too. Also, we show that a
umbilical non-degenerate hypersurface in a semi-Weyl manifold with torsion $M$ is also
umbilical with respect to any conformal-projective equivalent structure on $M$.

Then we focus on lightlike hypersurfaces of semi-Weyl manifolds with torsion and we describe similarities and differences with respect to the non-degenerate hypersurfaces. Lightlike submanifols are very interesting in relativity and their use is growing up in mathematical physics. The main difference between the  non-degenerate and the lightlike hypersurfaces is that, in the last case, the normal vector bundle intersects the tangent bundle, so a tangent vector cannot be decomposed uniquely into a  tangent and a normal component to the submanifold. In particular, the usual definition of the second fundamental form, Gauss and Weingarten formulas do not work for the lightlike case. For details we refer to the book \cite{DB}. In \cite{BT}, lightlike hypersurfaces of a statistical manifold are studied, with this motivation in mind, we generalize some results obtained for non-degenerate hypersurfaces to the case of lightlike hypersurfaces of a semi-Weyl manifold with torsion. In this situation, we need to consider an integrable screen distribution of the lightlike hypersurface. In particular, we prove the existence of the induced structure and the invariance under conformal-projective transformations.

Finally, we show that a semi-Weyl manifold with torsion can be realized by a non-degenerate affine distribution.

\section{Statistical and semi-Weyl manifolds with torsion}

Consider a pseudo-Riemannian manifold $(M,g)$. Throughout the whole paper, we will denote by $TM$ the tangent bundle of $M$, by $T^*M$ its cotangent bundle and by $\Gamma^{\infty}(TM)$ (respectively, by $\Gamma^{\infty}(T^*M)$) the smooth sections of $TM$ (respectively, of $T^*M$).
For an arbitrary affine connection $\nabla$ on $M$, we shall denote by $T^\nabla$ its torsion tensor and by $R^\nabla$ its curvature tensor, given, respectively, by
\[
T^\nabla(X,Y):=\nabla_XY-\nabla_Y X-[X,Y],
\]
\begin{equation*}\label{0}
R^\nabla(X,Y):=\nabla_X\nabla_Y-\nabla_Y\nabla_X-\nabla_{[X,Y]},
\end{equation*}
for $X, Y \in \Gamma^{\infty} (TM)$, where $[\cdot, \cdot]$ is the Lie bracket on $TM$. An affine connection is said to be \emph{torsion-free} if its torsion tensor is zero, and \emph{flat}, if its curvature tensor is zero.

\begin{definition} \cite{A1}
Let $(M,g)$ be a pseudo-Riemannian manifold and let $\nabla$ be a torsion-free affine connection on $M$.
Then $(M,g,\nabla)$ is called a \emph{statistical manifold} (on short, $SM$) if
\[
(\nabla _X g)(Y,Z)=(\nabla _Y g)(X,Z),
\]
for any $X, Y, Z \in \Gamma^{\infty} (TM)$.
\end{definition}

In 2007, Kurose introduced the notion of statistical manifold admitting torsion.

\begin{definition} \cite{k}
Let $(M,g)$ be a pseudo-Riemannian manifold and let $\nabla$ be an affine connection on $M$ with torsion tensor $T^{\nabla}$.
Then $(M,g,\nabla)$ is called a \emph{statistical manifold admitting torsion} (on short, $SMT$) if
\[
(\nabla _X g)(Y,Z)=(\nabla _Y g)(X,Z)-g(T^{\nabla}(X,Y),Z),
\]
for any $X, Y, Z \in \Gamma^{\infty} (TM)$.
\end{definition}

\begin{example}
If $g$ is a pseudo-Riemannian metric on $M$ and $f$ is a positive smooth function on $M$, then $(M,e^fg, \nabla^g+df\otimes I)$ is a $SMT$, where $\nabla^g$ is the Levi--Civita connection of $g$.
\end{example}

\begin{definition}
\cite{A1,la}
Let $(M,g)$ be a pseudo-Riemannian manifold. Two affine connections $\nabla$ and $\nabla^*$ on $M$ are said to be \emph{dual connections}
with respect to $g$ if
\begin{equation*}\label{a}
X(g(Y,Z))=g(\nabla_XY,Z)+g(Y,\nabla^*_XZ),
\end{equation*}
for any $X,Y,Z\in \Gamma^{\infty}(TM)$, and we call $(g,\nabla,\nabla^*)$ a \emph{dualistic structure on $M$}.
\end{definition}

As $g$ is symmetric, from the definition, we remark that $(\nabla^*)^*=\nabla$.
Notice that $\nabla = \nabla^*$ if and only if $\nabla$ is a metric connection, that is, $\nabla g=0$.
Moreover, if $\nabla$ is torsion-free, then $\nabla = \nabla^*$ if and only if $\nabla$ is the Levi--Civita connection of $g$.

\bigskip

Also, a dualistic structure $(g,\nabla, \nabla^*)$ on $M$ such that $\nabla$ and $\nabla^*$ are flat and torsion-free is called a
\emph{dually flat structure} on $M$ (and $(M, g,\nabla, \nabla^*)$ a \emph{dually flat manifold}).

\bigskip

From \cite{blna}, we state

\begin{proposition} \label{q1}
Let $\nabla$ and $\nabla^*$ be dual connections with respect to $g$. Then

(i) $R^\nabla=0 \Leftrightarrow R^{\nabla^*}=0$;

(ii) $T^{\nabla^*}=0 \Leftrightarrow (M,g,\nabla)$ is a $SMT$;

(iii) $T^\nabla=0 \Leftrightarrow (M,g,\nabla^*)$ is a $SMT$.

\end{proposition}

Recently, we introduced the notion of semi-Weyl manifold admitting torsion.

\begin{definition} \cite{blna}
Let $(M,g)$ be a pseudo-Riemannian manifold, let $\nabla$ be an affine connection on $M$ with torsion tensor $T^\nabla$ and let $\eta$ be a $1$-form.
Then $(M,g,\eta,\nabla)$ is called a \emph{semi-Weyl manifold admitting torsion} (on short, $SWMT$) if
\[
(\nabla_X g)(Y,Z)+\eta(X)g(Y,Z)=(\nabla_Y g)(X,Z)+\eta(Y)g(X,Z)-g(T^\nabla (X,Y),Z),
\]
for any $X,Y,Z \in  \Gamma^{\infty}(TM)$.
\end{definition}

\begin{example}
If $(M,g,\nabla)$ is a statistical manifold and $\eta$ is a nonzero $1$-form on $M$,
then $(M,g, \eta, \nabla+\eta\otimes I)$ is a $SWMT$.
\end{example}

\begin{definition}
\cite{NO,N0}
Let $(M,g)$ be a pseudo-Riemannian manifold and let $\eta$ be a $1$-form on $M$. Two affine connections $\nabla$ and $\nabla^*$ on $M$ are said to be \emph{semi-dual connections} (or \emph{generalized dual connections}) with respect to $(g,\eta)$ if
\begin{equation*}\label{b}
X(g(Y,Z))=g(\nabla_XY,Z)+g(Y,\nabla^*_XZ)-\eta(X)g(Y,Z),
\end{equation*}
for any $X,Y,Z\in \Gamma^{\infty}(TM)$, and we call $(g,\eta,\nabla,\nabla^*)$ a \emph{semi-dualistic structure on $M$}.
\end{definition}

From the definition, we remark that $(\nabla^*)^*=\nabla$.
Notice that if $\nabla$ is torsion-free, then $\nabla = \nabla^*$ if and only if $(M,g,\eta,\nabla)$ is a Weyl manifold.

\begin{remark}
If we denote by $\nabla^*_{g}$ the dual connection of $\nabla$ with respect to $g$ and by $\nabla^*_{(g,\eta)}$ the semi-dual connection of $\nabla$ with respect to $(g,\eta)$, then $\nabla^*_{g}=\nabla^*_{(g,\eta)} - \eta \otimes I$ and $\nabla=(\nabla^*_{(g,\eta)})^*_g+\eta \otimes I$.
\end{remark}

From \cite{blna}, we state

\begin{proposition} \label{q4}
Let $\nabla$ and $\nabla^*_{(g,\eta)}$ be semi-dual connections with respect to $(g,\eta)$ and let $\nabla^*_g$ be the dual connection of $\nabla$ with respect to $g$. Then

(i) $R^\nabla=0 \Leftrightarrow R^{\nabla^*_{(g,\eta)}}=0 \Leftrightarrow R^{\nabla^*_g}=0$;

(ii) $T^{\nabla^*_{(g,\eta)}}=0 \Leftrightarrow (M,g,\eta,\nabla)$ is a $SWMT$;

(iii) $T^\nabla=0 \Leftrightarrow (M,g,\eta,\nabla^*_{(g,\eta)})$ is a $SWMT$;


(iv) $(M,g,\eta,\nabla^*_{(g,\eta)})$ is a $SWMT$ $\Leftrightarrow$
$(M,g,\nabla^*_g)$ is a $SMT$.
\end{proposition}

\section{Conformal-projective transformations of $SMT$ and $SWMT$}

Let $\nabla$ be an affine connection on the pseudo-Riemannian manifold $(M,g)$ and consider the conformal-projective transformation of $(g,\nabla)$:
\begin{equation}\label{cc}
\widetilde{g}:=e^{\phi+\psi}g,
\end{equation}
\begin{equation}\label{dc}
\widetilde \nabla:=\nabla+d\phi\otimes I+I\otimes d\phi-g\otimes \nabla \psi,
\end{equation}
for $\phi$ and $\psi$ smooth functions on $M$, where $\nabla \psi$ denotes the gradient of $\psi$ with respect to the metric $g$.

By direct computations, we obtain
\begin{lemma}
\[
T^{\widetilde \nabla}=T^{\nabla},
\]
\[
(\widetilde \nabla_X\widetilde g)(Y,Z)-(\widetilde \nabla_Y\widetilde g)(X,Z)=e^{\phi+\psi}\Big((\nabla_Xg)(Y,Z)-(\nabla_Yg)(X,Z)\Big),
\]
for any $X,Y,Z\in \Gamma^{\infty}(TM)$.
\end{lemma}

The invariance under conformal-projective transformations is stated in the next proposition.

\begin{proposition} \label{p5z}
(i) $(M,g,\eta,\nabla)$ is a $SWMT$ if and only if $(M,\widetilde g,\eta,\widetilde \nabla)$ is a $SWMT$.
Moreover, the semi-dual connection $\widetilde\nabla_{(\widetilde g, \eta)}^*$ of $\widetilde \nabla$ with respect to $(\widetilde g, \eta)$ and the semi-dual connection $\nabla_{(g, \eta)}^*$ of $\nabla$ with respect to $(g,\eta)$ satisfy
$$\widetilde\nabla_{(\widetilde g, \eta)}^*=\nabla_{(g,\eta)}^*+d\psi\otimes I+I\otimes d\psi-g\otimes \nabla \phi.$$

(ii) $(M,g,\nabla)$ is a $SMT$ if and only if $(M,\widetilde g,\widetilde \nabla)$ is a $SMT$.
Moreover, the dual connection $\widetilde\nabla_{\widetilde g}^*$ of $\widetilde \nabla$ with respect to $\widetilde g$ and the dual connection $\nabla_g^*$ of $\nabla$ with respect to $g$ satisfy
$$\widetilde\nabla_{\widetilde g}^*=\nabla_g^*+d\psi\otimes I+I\otimes d\psi-g\otimes \nabla \phi.$$
\end{proposition}

\begin{proof}
The assertions follow from the definitions of the dual and semi-dual connections, (\ref{cc}) and (\ref{dc}).
\end{proof}

\begin{proposition}\label{p1z}
The curvatures of $(g,\nabla)$ and $(\widetilde g,\widetilde \nabla)$ satisfy
$$\widetilde R(X,Y)Z=R(X,Y)Z+Z(\phi)T^{\nabla}(X,Y)+\Big(X(\psi)g(Y,Z)-Y(\psi)g(X,Z)\Big)\nabla \psi$$
$$-\Big((d^\nabla g)(X,Y,Z)\Big)\nabla \psi$$
$$+\Big(X(Z(\phi))-g(\nabla_XZ,\nabla \phi)-X(\phi)Z(\phi)+g(X,Z)g(\nabla\phi,\nabla\psi)\Big)Y$$
$$-\Big(Y(Z(\phi))-g(\nabla_YZ,\nabla \phi)-Y(\phi)Z(\phi)+g(Y,Z)g(\nabla\phi,\nabla\psi)\Big)X$$
$$+g(X,Z)\nabla_Y\nabla\psi-g(Y,Z)\nabla_X\nabla\psi,$$
$$\widetilde \Ric(Y,Z)=\Ric(Y,Z)+Z(\phi)\trace(T_Y)+g(Y,Z)\Big(\Vert\nabla\psi\Vert^2-\Delta^{(\nabla,g)}(\psi)-(n-1)g(\nabla\phi,\nabla\psi)\Big)$$
$$+(n-1)Y(\phi)Z(\phi)-Y(\psi)Z(\psi)-g(T^{\nabla}(\nabla \psi,Y),Z)$$$$-(n-1)\Big((\nabla_Yg)(\nabla\phi,Z)+g(\nabla_Y\nabla\phi,Z)\Big)+(\nabla_Yg)(\nabla\psi,Z)+g(\nabla_Y\nabla\psi,Z)-(\nabla_{\nabla \psi}g)(Y,Z),$$
$$\widetilde \scal=e^{-(\phi+\psi)}\Big(\scal+\trace(T_{\nabla\phi})+\trace (T_{\nabla\psi})\Big)$$
$$+(n-1)e^{-(\phi+\psi)}\Big(\Vert\nabla\phi\Vert^2+\Vert\nabla\psi\Vert^2-\Delta^{(\nabla,g)}(\phi)-\Delta^{(\nabla,g)}(\psi)-ng(\nabla\phi,\nabla\psi)\Big)$$
$$-e^{-(\phi+\psi)}\Big((n-1)\trace ((\nabla g)(\nabla \phi))-\trace ((\nabla g)(\nabla \psi))+\trace (\nabla_{\nabla \psi} g)\Big),$$
for any $X,Y,Z\in \Gamma^{\infty}(TM)$, where $n=\dim(M)$, $T_Y(X):=T^{\nabla}(X,Y)$,
$$(d^\nabla g)(X,Y,Z):=(\nabla _X g)(Y,Z)-(\nabla _Y g)(X,Z)+g(T^{\nabla}(X,Y),Z)$$
and $$\Delta^{(\nabla,g)}(\phi):=\div^{(\nabla,g)}(\nabla \phi)=\sum_{i=1}^ng(\nabla_{E_i}\nabla\phi,E_i),$$ for $\{E_i\}_{1\leq i\leq n}$ an orthonormal frame field with respect to $g$.
\end{proposition}

\begin{proof}
We have:
\[
\widetilde \nabla_X\widetilde \nabla_YZ=\widetilde \nabla_X\Big(\nabla_YZ+d\phi(Y)Z+d\phi(Z)Y-g(Y,Z)\nabla \psi\Big)
\]
\[
=\widetilde \nabla_X\nabla_YZ+Y(\phi)\widetilde \nabla_XZ+X(Y(\phi))Z+Z(\phi)\widetilde \nabla_XY+X(Z(\phi))Y-g(Y,Z)\widetilde \nabla_X\nabla \psi-X(g(Y,Z))\nabla \psi
\]
\[
=\nabla_X\nabla_YZ+d\phi(X)\nabla_YZ+d\phi(\nabla_YZ)X-g(X,\nabla_YZ)\nabla \psi
\]
\[
+Y(\phi)\Big(\nabla_XZ+d\phi(X)Z+d\phi(Z)X-g(X,Z)\nabla \psi\Big)+X(Y(\phi))Z
\]
\[
+Z(\phi)\Big(\nabla_XY+d\phi(X)Y+d\phi(Y)X-g(X,Y)\nabla \psi\Big)+X(Z(\phi))Y
\]
\[
-g(Y,Z)\Big(\nabla_X\nabla \psi+d\phi(X)\nabla \psi+d\phi(\nabla \psi)X-g(X,\nabla \psi)\nabla \psi\Big)-X(g(Y,Z))\nabla \psi
\]
\[
=\nabla_X\nabla_YZ+X(\phi)\nabla_YZ+d\phi(\nabla_YZ)X-g(X,\nabla_YZ)\nabla \psi
\]
\[
+Y(\phi)\nabla_XZ+X(\phi)Y(\phi)Z+Y(\phi)Z(\phi)X-g(X,Z)Y(\phi)\nabla \psi+X(Y(\phi))Z
\]
\[
+Z(\phi)\nabla_XY+X(\phi)Z(\phi)Y+Y(\phi)Z(\phi)X-g(X,Y)Z(\phi)\nabla \psi+X(Z(\phi))Y
\]
\[
-g(Y,Z)\nabla_X\nabla \psi-g(Y,Z)X(\phi)\nabla \psi-g(Y,Z)g(\nabla \phi, \nabla \psi)X+g(Y,Z)X(\psi)\nabla \psi-X(g(Y,Z))\nabla \psi,
\]
for any $X,Y,Z\in \Gamma^{\infty}(TM)$, and by replacing it in
\[
\widetilde R(X,Y)Z:=\widetilde \nabla_X\widetilde \nabla_YZ-\widetilde \nabla_Y\widetilde \nabla_XZ-\widetilde \nabla_{[X,Y]}Z,
\]
we get the first relation.

From (\ref{cc}) and (\ref{dc}), taking into account that if $\{\widetilde E_i\}_{1\leq i\leq n}$ is an orthonormal frame field with respect to the pseudo-Riemannian metric $\widetilde g$, $\widetilde g(\widetilde E_i,\widetilde E_j)=\epsilon_i\delta_{ij}$, $\epsilon_i=\pm 1$, then
$\{E_i:=e^{\frac{\phi+\psi}{2}}\widetilde E_i\}_{1\leq i\leq n}$ is an orthonormal frame field with respect to $g$, from the definition of the Ricci tensor field, we get:
\[
\widetilde \Ric(Y,Z):=\sum_{i=1}^n\epsilon_i\widetilde g(\widetilde R(\widetilde E_i,Y)Z,\widetilde E_i)=\sum_{i=1}^n\epsilon_ig(\widetilde R(E_i,Y)Z,E_i)
\]
\[
=\sum_{i=1}^n\epsilon_ig(R(E_i,Y)Z,E_i)+Z(\phi)\sum_{i=1}^n\epsilon_ig(T^{\nabla}(E_i,Y),E_i)
\]
\[
+\sum_{i=1}^n\epsilon_i\Big(E_i(\psi)g(Y,Z)-Y(\psi)g(E_i,Z)-(\nabla_{E_i}g)(Y,Z)+(\nabla_Yg)(E_i,Z)-g(T^{\nabla}(E_i,Y),Z)\Big)g(\nabla \psi,E_i)
\]
\[
+\sum_{i=1}^n\epsilon_i\Big(E_i(Z(\phi))-g(\nabla_{E_i}Z,\nabla\phi)-E_i(\phi)Z(\phi)+g(E_i,Z)g(\nabla \phi,\nabla\psi)\Big)g(Y,E_i)
\]
\[
-\sum_{i=1}^n\Big(Y(Z(\phi))-g(\nabla_YZ,\nabla \phi)-Y(\phi)Z(\phi)+g(Y,Z)g(\nabla \phi,\nabla \psi)\Big)\epsilon_ig(E_i,E_i)
\]
\[
+\sum_{i=1}^n\epsilon_ig(E_i,Z)g(\nabla_Y\nabla \psi,E_i)-\sum_{i=1}^n\epsilon_ig(Y,Z)g(\nabla_{E_i}\nabla \psi,E_i)
\]
\[
=\Ric(Y,Z)+Z(\phi)\trace(T_Y)+g(Y,Z)\Big(\Vert\nabla\psi\Vert^2-\Delta^{(\nabla,g)}(\psi)-(n-1)g(\nabla\phi,\nabla\psi)\Big)
\]
\[+(n-1)Y(\phi)Z(\phi)-Y(\psi)Z(\psi)+g(\nabla_Y\nabla\psi,Z)-(\nabla_{\nabla \psi}g)(Y,Z)+(\nabla_Yg)(\nabla\psi,Z)
\]
\[
-g(T^{\nabla}(\nabla\psi,Y),Z)-(n-1)\Big((\nabla_Yg)(\nabla\phi,Z)+g(\nabla_Y\nabla\phi,Z)\Big).
\]

Also:
\[
\widetilde \scal:=\sum_{i=1}^n\epsilon_i\widetilde \Ric(\widetilde E_i,\widetilde E_i)=e^{-(\phi+\psi)}\sum_{i=1}^n\epsilon_i\widetilde \Ric(E_i,E_i)
\]
\[
=e^{-(\phi+\psi)}\Big[\sum_{i=1}^n\epsilon_i\Ric(E_i,E_i)+\sum_{i=1}^n\epsilon_iE_i(\phi)\trace(T_{E_i})+n\Big(\Vert\nabla\psi\Vert^2-\Delta^{(\nabla,g)}(\psi)-(n-1)g(\nabla\phi,\nabla\psi)\Big)
\]
\[
+(n-1)\sum_{i=1}^n\epsilon_iE_i(\phi)E_i(\phi)-\sum_{i=1}^n\epsilon_iE_i(\psi)E_i(\psi)-\sum_{i=1}^n\epsilon_ig(T^{\nabla}(\nabla\psi,E_i),E_i)
\]
\[
-(n-1)\sum_{i=1}^n\epsilon_i\Big((\nabla_{E_i}g)(\nabla\phi,E_i)+g(\nabla_{E_i}\nabla\phi,E_i)\Big)+
\sum_{i=1}^n\epsilon_i\Big((\nabla_{E_i}g)(\nabla\psi,E_i)+g(\nabla_{E_i}\nabla\psi,E_i)\Big)
\]
\[
-\sum_{i=1}^n\epsilon_i(\nabla_{\nabla \psi}g)(E_i,E_i)\Big]
\]
\[
=e^{-(\phi+\psi)}\Big[\scal+\trace(T_{\nabla \phi})+n\Big(\Vert\nabla\psi\Vert^2-\Delta^{(\nabla,g)}(\psi)-(n-1)g(\nabla\phi,\nabla\psi)\Big)
\]
\[
+(n-1)\Vert\nabla \phi\Vert^2-\Vert\nabla \psi\Vert^2+\trace(T_{\nabla \psi})
\]
\[
-(n-1)\trace((\nabla g)(\nabla\phi))-(n-1)\Delta^{(\nabla,g)}(\phi)+
\trace((\nabla g)(\nabla\psi))+\Delta^{(\nabla,g)}(\psi)-\trace(\nabla_{\nabla \psi}g)\Big]
\]
and we obtain the conclusions.
\end{proof}

Now we can state
\begin{proposition}\label{p2z}
\[
\widetilde \Ric(Y,Z)-\widetilde \Ric(Z,Y)=\Ric(Y,Z)-\Ric(Z,Y)
\]
\[
+Z(\phi)\trace(T_Y)-Y(\phi)\trace(T_Z)
-g(T^{\nabla}(\nabla \psi,Y),Z)+g(T^{\nabla}(\nabla \psi,Z),Y)
\]
\[
-(n-1)\Big((\nabla_Yg)(Z,\nabla\phi)-(\nabla_Zg)(Y,\nabla\phi)+g(\nabla_Y\nabla\phi,Z)-g(\nabla_Z\nabla\phi,Y)\Big)
\]
\[
+(\nabla_Yg)(Z,\nabla\psi)-(\nabla_Zg)(Y,\nabla\psi)+g(\nabla_Y\nabla\psi,Z)-g(\nabla_Z\nabla\psi,Y),
\]
for any $Y,Z\in \Gamma^{\infty}(TM)$.
\end{proposition}

\begin{remark} \label{r2}
For any smooth function $f$ on $M$, remark that
$$(\nabla_Yg)(Z,\nabla f)-(\nabla_Zg)(Y,\nabla f)=Y(Z(f))-Z(Y(f))-g(\nabla_YZ-\nabla_ZY,\nabla f)$$$$+g(Y,\nabla_Z\nabla f)-g(Z,\nabla_Y\nabla f)$$
$$=[Y,Z](f)-g(T^{\nabla}(Y,Z), \nabla f)-g([Y,Z],\nabla f)+g(Y,\nabla_Z\nabla f)-g(Z,\nabla_Y\nabla f)$$
$$=-g(T^{\nabla}(Y,Z), \nabla f)+g(Y,\nabla_Z\nabla f)-g(Z,\nabla_Y\nabla f),$$
for any $Y,Z\in \Gamma^{\infty}(TM)$.
\end{remark}

Therefore, by means of Remark \ref{r2}, Proposition \ref{p2z} can be restated as
\begin{proposition}\label{pk}
\[
\widetilde \Ric(Y,Z)-\widetilde \Ric(Z,Y)=\Ric(Y,Z)-\Ric(Z,Y)
+Z(\phi)\trace(T_Y)-Y(\phi)\trace(T_Z)
\]
\[
+(n-1)g(T^{\nabla}(Y,Z),\nabla \phi)
-g(T^{\nabla}(Y,Z),\nabla \psi)-g(T^{\nabla}(Z,\nabla \psi),Y)-g(T^{\nabla}(\nabla \psi,Y),Z),
\]
for any $Y,Z\in \Gamma^{\infty}(TM)$.
\end{proposition}

\begin{remark} \label{r1}
From Remark \ref{r2}, we also deduce that:

(i) if $(M,g,\eta,\nabla)$ is a $SWMT$, then
$$g(\nabla_Y\nabla f,Z)-g(\nabla_Z\nabla f,Y)=\eta(Y)Z(f)-\eta(Z)Y(f),$$
for any $Y,Z\in \Gamma^{\infty}(TM)$;

(ii) if $(M,g,\nabla)$ is a $SMT$, then
$$g(\nabla_Y\nabla f,Z)=g(\nabla_Z\nabla f,Y),$$
for any $Y,Z\in \Gamma^{\infty}(TM)$.
\end{remark}



From Proposition \ref{p1z}, for $\phi=0$, we get
\begin{corollary}\label{cg}
If $\widetilde{g}:=e^{\psi}g$ and
$\widetilde \nabla:=\nabla-g\otimes \nabla \psi$, then
\[
\widetilde R(X,Y)Z=R(X,Y)Z+\Big(X(\psi)g(Y,Z)-Y(\psi)g(X,Z)\Big)\nabla \psi
\]
\[
-\Big((\nabla_Xg)(Y,Z)-(\nabla_Yg)(X,Z)+g(T^{\nabla}(X,Y),Z)\Big)\nabla \psi+g(X,Z)\nabla_Y\nabla\psi-g(Y,Z)\nabla_X\nabla\psi,
\]
\[
\widetilde \Ric(Y,Z)=\Ric(Y,Z)+g(Y,Z)\Big(\Vert\nabla\psi\Vert^2-\Delta^{(\nabla,g)}(\psi)\Big)-Y(\psi)Z(\psi)
\]
\[
-g(T^{\nabla}(\nabla\psi,Y),Z)+(\nabla_Yg)(\nabla\psi,Z)+g(\nabla_Y\nabla\psi,Z)-(\nabla_{\nabla \psi}g)(Y,Z),
\]
\[
\widetilde \scal=e^{-\psi}\Big[\scal+\trace (T_{\nabla\psi})-\trace (\nabla_{\nabla \psi} g)+\trace ((\nabla g)(\nabla \psi))+(n-1)\Big(\Vert\nabla\psi\Vert^2-\Delta^{(\nabla,g)}(\psi)\Big)\Big],
\]
for any $X,Y,Z\in \Gamma^{\infty}(TM)$.
\end{corollary}

In this case, Proposition \ref{pk} becomes
\begin{corollary}\label{n}
If $\widetilde{g}:=e^{\psi}g$ and
$\widetilde \nabla:=\nabla-g\otimes \nabla \psi$, then
\[
\widetilde \Ric(Y,Z)-\widetilde \Ric(Z,Y)=\Ric(Y,Z)-\Ric(Z,Y)
\]
\[
-g(T^{\nabla}(Y,Z),\nabla \psi)-g(T^{\nabla}(Z,\nabla \psi),Y)-g(T^{\nabla}(\nabla \psi,Y),Z),
\]
for any $Y,Z\in \Gamma^{\infty}(TM)$.
\end{corollary}

\begin{corollary}\label{ps}
Let $\widetilde{g}:=e^{\psi}g$,
$\widetilde \nabla:=\nabla-g\otimes \nabla \psi$ and let $\eta$ be a $1$-form on $M$.
If $(M,g,\eta,\nabla)$ is a $SWMT$ (or a $SMT$), then
$$\widetilde \Ric(Y,Z)-\widetilde \Ric(Z,Y)=\Ric(Y,Z)-\Ric(Z,Y),$$
for any $Y,Z\in \Gamma^{\infty}(TM)$.
\end{corollary}

\begin{proof}
From Corollary \ref{cg}, we have:
\[
\widetilde R(X,Y)Z=R(X,Y)Z+\Big(X(\psi)g(Y,Z)-Y(\psi)g(X,Z)\Big)\nabla \psi
\]
\[
-\Big(\eta(Y)g(X,Z)-\eta(X)g(Y,Z)\Big)\nabla \psi+g(X,Z)\nabla_Y\nabla\psi-g(Y,Z)\nabla_X\nabla\psi,
\]
which by taking the trace, gives
\[
\widetilde \Ric(Y,Z)=\Ric(Y,Z)+g(Y,Z)\Big(\Vert\nabla\psi\Vert^2-\Delta^{(\nabla,g)}(\psi)+\eta(\nabla\psi)\Big)
\]
\[
-\Big(Y(\psi)+\eta(Y)\Big)Z(\psi)
+g(\nabla_Y\nabla\psi,Z)
\]
and by using Remark \ref{r1}, we find
\[
\widetilde \Ric(Y,Z)-\widetilde \Ric(Z,Y)=\Ric(Y,Z)-\Ric(Z,Y)
\]
\[
-\eta(Y)Z(\psi)+\eta(Z)Y(\psi)+g(\nabla_Y\nabla\psi,Z)-g(\nabla_Z\nabla\psi,Y)=\Ric(Y,Z)-\Ric(Z,Y).
\]

Also:
\[
\widetilde \scal=e^{-\psi}\Big[\scal+(n-1)\Big(\Vert\nabla\psi\Vert^2-\Delta^{(\nabla,g)}(\psi)+\eta(\nabla \psi)\Big)\Big].
\]

For $\eta=0$, we get the corresponding relations for $SMT$.
\end{proof}

From Corollaries \ref{n} and \ref{ps}, we deduce
\begin{corollary}
Let $\widetilde{g}:=e^{\psi}g$,
$\widetilde \nabla:=\nabla-g\otimes \nabla \psi$ and let $\eta$ be a $1$-form on $M$.
If $(M,g,\eta,\nabla)$ is a $SWMT$ (or a $SMT$), then
\[
g(T^{\nabla}(Y,Z),\nabla \psi)+g(T^{\nabla}(Z,\nabla \psi),Y)+g(T^{\nabla}(\nabla \psi,Y),Z)=0,
\]
for any $Y,Z\in \Gamma^{\infty}(TM)$.
\end{corollary}


\begin{definition} \cite{kui}
A pair $(g,\nabla)$ of an affine connection and a pseudo-Riemannian metric on $M$ is said to be a \emph{conformally flat structure} if there exists a smooth function $\psi$ such that the connection $\widetilde \nabla:=\nabla-g\otimes \nabla \psi$
is flat in a neighborhood of any point of $M$.
\end{definition}

Thus, from Corollaries \ref{cg} and \ref{n}, we can state

\begin{corollary}
(i) If $(M,g,\nabla)$ is a $SWMT$ such that $(g,\nabla)$ is a conformally flat structure on $M$ with $\widetilde \nabla:=\nabla-g\otimes \nabla \psi$ flat in a neighborhood of any point, then, for any $X,Y, Z\in \Gamma^{\infty}(TM)$, we get
$$R(X,Y)Z=-\Big(X(\psi)g(Y,Z)-Y(\psi)g(X,Z)\Big)\nabla \psi-g(X,Z)\nabla_Y\nabla\psi+g(Y,Z)\nabla_X\nabla\psi$$
$$+\Big(\eta(Y)g(X,Z)-\eta(X)g(Y,Z)\Big)\nabla \psi,$$
$$\Ric(Y,Z)=-g(Y,Z)\Big(\Vert\nabla\psi\Vert^2-\Delta^{(\nabla,g)}(\psi)+\eta(\nabla\psi)\Big)+\Big(Y(\psi)+\eta(Y)\Big)Z(\psi)-g(\nabla_Y\nabla\psi,Z)$$
$$\scal=-(n-1)\Big(\Vert\nabla\psi\Vert^2-\Delta^{(\nabla,g)}(\psi)+\eta(\nabla\psi)\Big).$$

In particular, the Ricci tensor field of $\nabla$ is symmetric.

(ii) If $(M,g,\nabla)$ is a $SMT$ such that $(g,\nabla)$ is a conformally flat structure on $M$ with $\widetilde \nabla:=\nabla-g\otimes \nabla \psi$ flat in a neighborhood of any point, then, for any $X,Y, Z\in \Gamma^{\infty}(TM)$, we get
$$R(X,Y)Z=-\Big(X(\psi)g(Y,Z)-Y(\psi)g(X,Z)\Big)\nabla \psi-g(X,Z)\nabla_Y\nabla\psi+g(Y,Z)\nabla_X\nabla\psi,$$
$$\Ric(Y,Z)=-g(Y,Z)\Big(\Vert\nabla\psi\Vert^2-\Delta^{(\nabla,g)}(\psi)\Big)+Y(\psi)Z(\psi)-g(\nabla_Y\nabla\psi,Z),$$
$$\scal=-(n-1)\Big(\Vert\nabla\psi\Vert^2-\Delta^{(\nabla,g)}(\psi)\Big).$$

In particular, the Ricci tensor field of $\nabla$ is symmetric.
\end{corollary}

\section{Pseudo-Riemannian submanifolds of $SMT$ and $SWMT$}

Let $(M,g)$ be an $n$-dimensional pseudo-Riemannian manifold.

\begin{definition}
A smooth submanifold $M'$ of $M$ is called a \emph{pseudo-Riemannian submanifold}, or \emph{non-degenerate submanifold}, of $(M,g)$ if the induced tensor $g':=g_{\vert TM'\times TM'}$ is a pseudo-Riemannian metric on $M'$.
\end{definition}

Let $M'$ be a pseudo-Riemannian submanifold of $(M,g)$. Then the tangent space of $M$, in any point $x\in M'$, can be orthogonally decomposed, with respect to $g$, into the direct sum
\[
T_xM=T_xM'\oplus T^{\bot}_xM'.
\]

We shall prove that a pseudo-Riemannian submanifold $M'$ of a $SWMT$ (respectively, $SMT$) $M$ is also a $SWMT$ (respectively, $SMT$) with the induced structure and that the conformal-projective equivalence of two such structures is preserved for the induced structures on the submanifold.

\begin{proposition} \label{p10}
If $M'$ is a pseudo-Riemannian submanifold of a $SWMT$ $(M,g,\eta,\nabla)$, then $(M',g',\eta',\nabla')$ is also a $SWMT$, where
$\eta'$ and $g'$ are the induced tensors on $M'$ and $\nabla'_XY$, $X,Y\in \Gamma^{\infty}(TM')$, is the component of $\nabla_XY$ tangent to $M'$. In this case, we call $(g',\eta',\nabla')$ the \emph{induced structure} by $(g,\eta,\nabla)$ on $M'$. Moreover, the semi-dual connection $(\nabla')^*$ of the induced connection $\nabla'$ with respect to $(g',\eta')$ is the induced connection $(\nabla^*)'$ of the semi-dual connection $\nabla^*$ of $\nabla$ with respect to $(g,\eta)$.
\end{proposition}

\begin{proof}
Just remark that
\[
T^{\nabla}(X+U,Y+V)=T^{\nabla}(X,Y)+T^{\nabla}(X,V)+T^{\nabla}(U,Y)+T^{\nabla}(U,V),
\]
for any $X,Y\in \Gamma^{\infty}(TM')$ and $U,V\in \Gamma^{\infty}(T^{\bot}M')$, hence $$g(T^{\nabla}(X,Y),Z)=g'(T^{\nabla'}(X,Y),Z),$$ for any $X,Y,Z\in \Gamma^{\infty}(TM')$. Also
\[
(\nabla'_Xg')(Y,Z):=X(g'(Y,Z))-g'(\nabla'_XY,Z)-g'(Y,\nabla'_XZ)
\]
\[
=X(g(Y,Z))-g(\nabla_XY,Z)-g(Y,\nabla_XZ):=(\nabla_Xg)(Y,Z),
\]
for any $X,Y,Z\in \Gamma^{\infty}(TM')$, hence
\[
(\nabla'_Xg')(Y,Z)-(\nabla'_Yg')(X,Z)=(\nabla_Xg)(Y,Z)-(\nabla_Yg)(X,Z)
\]
\[
=\eta(Y)g(X,Z)-\eta(X)g(Y,Z)-g(T^{\nabla}(X,Y),Z)
\]
\[
=\eta'(Y)g'(X,Z)-\eta'(X)g'(Y,Z)-g'(T^{\nabla'}(X,Y),Z)
\]
and we get the first statement.

Moreover, we remark that
\[
X(g'(Y,Z))=X(g(Y,Z))=g(\nabla_XY,Z)+g(Y,\nabla^*_XZ)-\eta(X)g(Y,Z)
\]
\[
=g'(\nabla'_XY,Z)+g'(Y,(\nabla^*)'_XZ)-\eta'(X)g'(Y,Z),
\]
for any $X,Y,Z\in \Gamma^{\infty}(TM')$, hence the conclusion.
\end{proof}

We recall the following.

\begin{definition}
A smooth hypersurface $M'$ of $M$ is called a \emph{lightlike hypersurface} of $(M,g)$ if the induced tensor $g':=g_{\vert TM'\times TM'}$ is degenerate.
\end{definition}

\begin{remark} We remark that Proposition \ref{p10} is not true for smooth submanifolds which are not pseudo-Riemannian submanifolds, for example it is not true for  lightlike hypersurfaces \cite{BT}.
\end{remark}

\begin{proposition}
Let $M'$ be an $m$-dimensional pseudo-Riemannian submanifold of an $n$-dimensional pseudo-Riemannian manifold $(M,g)$. If $(M,g,\eta,\nabla)$ and $(M,\widetilde g,\eta,\widetilde \nabla)$ are conformal-projective equivalent $SWMT$, then
the induced structures on $M'$, $(g',\eta',\nabla')$ and $(\widetilde g',\eta',\widetilde \nabla')$, are conformal-projective equivalent, too.
\end{proposition}

\begin{proof}
Obviously, the induced metrics $g'$ and $\widetilde g'$ on $M'$ are conformal equivalent, namely,
\[
\widetilde g'=e^{\phi'+\psi'}g',
\]
where $\phi'$ and $\psi'$ denotes the restrictions of $\phi$ and $\psi$ to $M'$, respectively.

Let $X,Y\in  \Gamma^{\infty}(TM')$. Then
\[
\widetilde \nabla'_XY+(\widetilde \nabla_XY)^{\bot}=\widetilde \nabla_XY=\nabla_XY+d\phi(X)Y+d\phi(Y)X-g(X,Y)\nabla\psi
\]
\[
=\nabla'_XY+(\nabla_XY)^{\bot}+d\phi'(X)Y+d\phi'(Y)X-g'(X,Y)\nabla'\psi'-g'(X,Y)\sum_{i=m+1}^n\epsilon_ig(\nabla \psi,N_i)N_i,
\]
where $\{N_i\}_{m+1\leq i\leq n}$ is an orthonormal frame field on $T^{\bot}M'$, $\epsilon_i=\pm 1$. By identifying the tangential components, we infer
\[
\widetilde \nabla'_XY=\nabla'_XY+d\phi'(X)Y+d\phi'(Y)X-g'(X,Y)\nabla'\psi',
\]
hence the conclusion.
\end{proof}

\bigskip


Let $(M',g',\eta',\nabla')$ be a non-degenerate hypersurface of a $SWMT$ $(M,g,\eta,\nabla)$, with the induced structure, and denote by $N$ the unit normal vector field with respect to $g$, $g(N,N)=\epsilon=\pm 1$.
Then
\[
\nabla_XY=\nabla'_XY+\alpha(X,Y)N, \ \ \nabla_XN=-B(X)+\tau(X)N,
\]
\[
\nabla^*_XY=(\nabla')^*_XY+\alpha^*(X,Y)N, \ \ \nabla^*_XN=-B^*(X)+\tau^*(X)N,
\]
where
\[
\alpha(X,Y):=\epsilon g(\nabla_XY,N), \ \tau(X):=\epsilon g(\nabla_XN,N),
\]
\[
\alpha^*(X,Y):=\epsilon g(\nabla^*_XY,N), \ \tau^*(X):=\epsilon g(\nabla^*_XN,N),
\]
and we define
\[ \ \beta(X,Y):=g(B(X),Y)=-g(\nabla_X N,Y),
\]
\[ \ \beta^*(X,Y):=g(B^*(X),Y)=-g(\nabla^*_X N,Y),
\]
for any $X,Y\in \Gamma^{\infty}(TM')$.


\begin{proposition} Let $(M',g',\eta',\nabla')$ be a non-degenerate hypersurface of a $SWMT$ $(M,g,\eta,\nabla)$, with the induced structure. Then the $(0,2)$-tensor field $\beta$ (defined on $M'$) is symmetric. Moreover, $(\epsilon \alpha,\beta)=(\beta^*,\epsilon \alpha^*)$.
\end{proposition}
\begin{proof} We have
\[
\beta(X,Y)-\beta(Y,X)=-g(\nabla_XN,Y)+g(\nabla_YN,X)
\]
\[
=-\Big(X(g(N,Y))-g(N,\nabla^*_XY)+\eta(X)g(N,Y)\Big)+\Big(Y(g(N,X))-g(N,\nabla^*_YX)+\eta(Y)g(N,X)\Big)
\]
\[=g(N,\nabla^*_XY)-g(N,\nabla^*_YX)=-g(N,T^{\nabla^*}(X,Y)),
\]
for any $X,Y\in \Gamma^{\infty}(TM')$, where $\nabla^*$ is the semi-dual connection of $\nabla$ with respect to $(g,\eta)$. In particular, the first statement follows from Proposition \ref{q4}.

Moreover
\[
\beta(X,Y)=g(N,\nabla^*_XY)=\epsilon \alpha^*(X,Y),
\]
hence
\[
(\epsilon \alpha,\beta)=(\beta^*,\epsilon \alpha^*).
\]
\end{proof}

We pose the following.

\begin{definition} A point $x$ of a pseudo-Riemannian hypersurface $M'$ in a $SWMT$ $(M,g,\eta,\nabla)$ is called a \emph{umbilical point} if there exists a real constant $c$ such that $\beta_x=cg'_x$. Moreover, $M'$ is said to be a \emph{umbilical hypersurface} of $M$ if there exists a smooth function $f$ on $M'$ such that $\beta=fg'$.
\end{definition}

Now we can prove the following.

\begin{proposition}
The umbilical points of a pseudo-Riemannian hypersurface in a $SWMT$, with the induced structure, are preserved by con\-for\-mal-projective changes.
\end{proposition}

\begin{proof}
Let $M'$ be a non-degenerate hypersurface of a $SWMT$ $(M,g,\eta,\nabla)$ and let
$(g',\eta',\nabla')$ be the induced structure on $M'$.
Remark that the unit normal vector field $\widetilde N$ with respect to $\widetilde g$ equals
to $\widetilde N=e^{-\frac{\phi+\psi}{2}}N$.
Then, for any $X,Y\in \Gamma^{\infty}(TM)$, we have
\[
\widetilde \beta(X,Y)=-\widetilde g(\widetilde \nabla_X\widetilde N,Y)=-e^{\frac{\phi+\psi}{2}}g(\nabla_XN+d\phi(X)N+d\phi(N)X-g(X,N)\nabla\psi, Y)
\]
\[
-e^{\phi+\psi}X(e^{-\frac{\phi+\psi}{2}})g(N,Y)
\]
\[
=-e^{\frac{\phi+\psi}{2}}\Big(g(\nabla_XN, Y)+d\phi(N)g(X,Y)\Big)= e^{\frac{\phi+\psi}{2}}\Big(\beta(X,Y)-d\phi(N)g'(X,Y)\Big),
\]
hence
\[
\widetilde \beta=e^{\frac{\phi+\psi}{2}}(\beta-d\phi(N)g')
\]
and we immediately get the conclusions.
\end{proof}

As a consequence, we obtain
\begin{corollary}
Let $M'$ be a $m$-dimensional 
pseudo-Riemannian submanifold of a $n$-dimensional pseudo-Riemannian manifold $(M,g)$. If $(M,g,\eta,\nabla)$ and $(M,\widetilde g,\eta,\widetilde \nabla)$ are conformal-projective equivalent $SWMT$ and $M'$ is a 
umbilical hypersurface of $(M,g,\eta,\nabla)$, then
$M'$ is a 
umbilical hypersurface of $(M,\widetilde g,\eta,\widetilde \nabla)$, too.
\end{corollary}

\begin{proposition}\label{p4z}
The Riemann curvatures of $(g,\nabla)$ and $(g',\nabla')$ satisfy
\[
R^{\nabla}(X,Y)Z=R^{\nabla'}(X,Y)Z-\Big(\alpha(Y,Z)B(X)-\alpha(X,Z)B(Y)\Big)
\]
\[+\Big((\nabla'_X\alpha)(Y,Z)+\alpha(Y,Z)\tau(X)-(\nabla'_Y\alpha)(X,Z)-\alpha(X,Z)\tau(Y)+\alpha(T^{\nabla'}(X,Y),Z)\Big)N,
\]
for any $X,Y,Z\in \Gamma^{\infty}(TM')$.
\end{proposition}

\begin{proof}
We have:
\[
\nabla_X\nabla_YZ=\nabla_X(\nabla'_YZ+\alpha(Y,Z)N)=\nabla_X\nabla'_YZ+\alpha(Y,Z)\nabla_XN+X(\alpha(Y,Z))N=
\]
\[
=\nabla'_X\nabla'_YZ+\alpha(X,\nabla'_YZ)N+\alpha(Y,Z)(-B(X)+\tau(X)N)+X(\alpha(Y,Z))N=
\]
\[
=\nabla'_X\nabla'_YZ+\alpha(X,\nabla'_YZ)N-\alpha(Y,Z)B(X)+\alpha(Y,Z)\tau(X)N+
\]
\[
+\left((\nabla'_X\alpha)(Y,Z)+\alpha(\nabla'_XY,Z)+\alpha(Y,\nabla'_XZ)\right)N,
\]
for any $X,Y,Z\in \Gamma^{\infty}(TM')$, and by replacing it in
\[
R^{\nabla}(X,Y)Z:=\nabla_X\nabla_YZ-\nabla_Y\nabla_XZ-\nabla_{[X,Y]}Z,
\]
we get the relation.
\end{proof}

\begin{proposition}
Let $M'$ be a non-degenerate hypersurface of a $SWMT$ $(M,g,\eta,\nabla)$ with the induced structure $(g',\eta',\nabla')$.
If $M'$ is umbilical with $\beta=fg'$, for $f$ a smooth function on $M'$ and $R^{{\nabla}^*}(X,Y)Z=0$, for any $X,Y,Z\in \Gamma^{\infty}(TM')$, then
\[
R^{(\nabla')^*}(X,Y)Z=\epsilon f\Big(g'(Y,Z)B^*(X)-g'(X,Z)B^*(Y)\Big),
\]
\[
df+f(\tau^*-\eta')=0.
\]
\end{proposition}

\begin{proof}
Replacing the semi-dual connection $\nabla^*$ of $\nabla$ with respect to $(g,\eta)$ and $\epsilon \beta^*$ to $\nabla$ and $\alpha$, respectively, in Proposition \ref{p4z}, taking into account that
\[
(\nabla'_X\beta)(Y,Z)=X(f)g'(Y,Z)+f(\nabla'_Xg')(Y,Z),
\]
and since $T^{\nabla^*}=0$, we infer
\[
R^{\nabla^*}(X,Y)Z=R^{(\nabla')^*}(X,Y)Z-\epsilon\Big(fg'(Y,Z)B^*(X)-fg'(X,Z)B^*(Y)\Big)
\]
\[+\epsilon\Big (X(f)g'(Y,Z)+f(\nabla'_Xg')(Y,Z)+fg'(Y,Z)\tau^*(X)
\]
\[
-Y(f)g'(X,Z)-f(\nabla'_Yg')(X,Z)-fg'(X,Z)\tau^*(Y)\Big)N
\]
\[
=R^{(\nabla')^*}(X,Y)Z-\epsilon f\Big(g'(Y,Z)B^*(X)-g'(X,Z)B^*(Y)\Big)
\]
\[
+\epsilon \Big((df+f\tau^*-f\eta')(X)g'(Y,Z)-(df+f\tau^*-f\eta')(Y)g'(X,Z)\Big)N,
\]
for any $X,Y,Z\in \Gamma^{\infty}(TM')$, which implies the conclusion.
\end{proof}

\section{Lightlike hypersurfaces of $SMT$ and $SWMT$}

In this section we describe some details concerning the geometry of lightlike hypersurfaces of $SMT$ and $SWMT$.\\

Let $(M,g)$ be an $(n+2)$-dimensional pseudo-Riemannian manifold with the \emph{index of negativity of} $g$ greater than $0$ and less than $n+2$.
If $(M',g')$ is a lightlike hypersurface of $(M,g)$, then there exists a non zero $\zeta\in \Gamma^{\infty}(TM')$ such that
$g'(\zeta,X)=0$, for any $X\in  \Gamma^{\infty}(TM')$. Let $Rad (TM')$ be the distribution (of rank $1$), called the \emph{radical distribution}, defined by:
$$Rad (TM'):=\{\zeta\in \Gamma ^{\infty}(TM') :g'(\zeta,X)=0, \ \textrm{for any} \ X\in \Gamma ^{\infty}(TM')\}.$$
Let $S(TM')$ be a complementary vector bundle of $Rad(TM')$ in $TM'$, called a \emph{screen distribution}.
Then we have that $TM'$ decomposes into an orthogonal direct sum
$$TM'=S(TM')\oplus_{\perp} Rad (TM').$$
In particular, $S(TM')$ is a non-degenerate distribution and we have
$$T M_{\vert{M'}} = S(TM')\oplus_{\perp} (S(TM'))^\perp,$$
where $(S(TM'))^\perp$ is the orthogonal complement of $S(TM')$ in $T M_{\vert M'}$.

The following fact is known.
\begin{theorem}\cite{DB}
Let $(M', g', S(TM'))$ be a lightlike hypersurface of a pseudo-Rie\-man\-ni\-an
manifold $(M ,g)$. Then, there exists a unique vector bundle $tr(TM')$ of rank $1$ over
$M'$, such that, on any coordinate neighborhood
$U\subset M'$, there exists a unique $N\in\Gamma^{\infty}(tr(TM')_{\vert U})$ satisfying:
$$g(N, \xi) = 1, \, g(N,N) =g(N,W) = 0,  \ \textrm{for any} \ W \in \Gamma^{\infty}(S(TM')_{\vert U}),$$
where $\{\xi\}$ is a local basis for $\Gamma^{\infty}(Rad(TM')_{\vert U})$.
In this case, $tr(TM')$ is called the \emph{lightlike transversal vector bundle} of $M'$ with respect to $S(TM')$.
\end{theorem}
Hence we have the following
\[
T M_{\vert{M'}} =  S(TM')\oplus_{\perp}((T^\perp M') \oplus tr(TM'))
\]
\[= S(TM')\oplus_{\perp}(Rad(TM') \oplus tr(TM'))=TM'\oplus tr(TM').\]

Now let us suppose that $(M,g,\eta,\nabla)$ is a $SWMT$ and let $(M',g')$ be a lightlike hypersurface. Locally, on a coordinate neighborhood
$U\subset M'$, we have Gauss and Weingarten formulas:
$$\nabla_X Y=\nabla'_X Y+\alpha(X,Y)N,$$
$$\nabla_X N=-B(X)+\tau (X)N,$$
for any $X,Y\in\Gamma^\infty(TM'_{\vert U})$, where $\nabla'$ represents the induced connection on $M'$, $\nabla'_X Y\in \Gamma^{\infty}(TM'_{\vert U})$, $\alpha$ is the second fundamental form, $B$ is the Weingarten operator and $\tau(X)=g(\nabla_X N,\xi)$.

The following result generalizes Proposition \ref{p10} to screen bundle and generalizes the analogous result for lightlike submanifolds of a statistical manifold given in \cite{BT}.

\begin{proposition}
Let $(M',g')$ be a lightlike hypersurface of a $SWMT$ $(M,g,\eta,\nabla)$ such that the screen distribution $S(TM')$ is integrable.
Then $(g',\eta', \bar \nabla)$ is a semi-Weyl structure with torsion on $S(TM')$, where $\eta'$ is the induced $1$-form and
$\bar \nabla$ is the induced connection on sections of the screen distribution.
\end{proposition}

\begin{proof}
By using previous notations, let $P$ be the projection of $\Gamma^\infty(TM')$ on $\Gamma^\infty(S(TM'))$. For $X\in\Gamma^\infty(TM')$, we can write
$$X=PX+\gamma(X)\xi,$$
where $\gamma$ is a $1$-form given by $$\gamma(X):=g(X,N).$$
We have:
$$\nabla_{PX}PY={\bar \nabla }_{PX} PY+h(PX,PY),$$
for any $X,Y\in\Gamma^\infty(TM')$, where $h$ is the second fundamental form of $S(TM')$ in $TM$.
Hence
\[(\bar \nabla_{PX} g')(PY,PZ)=(PX)(g(PY,PZ))-g(\nabla_{PX} PY,PZ)-g(PY,\nabla_{PX}PZ)\]
\[=(\nabla_{PX} g)(PY,PZ),\]
for any $X,Y,Z\in\Gamma^\infty(TM')$.
Moreover, the torsion of $\bar \nabla$,
$$T^{\bar \nabla}(X,Y):=\bar \nabla_{X}Y-\bar \nabla_{Y}X-[X,Y],$$
satisfies
$$g'(T^{\bar \nabla}(PX,PY),PZ)=g(\bar \nabla_{PX}PY-\bar \nabla_{PY}PX-[PX,PY],PZ)$$
$$=g(T^\nabla(PX,PY),PZ).$$
Then
$$(\bar \nabla_{PX} g')(PY,PZ)+\eta(PX)g'(PY,PZ)-(\bar \nabla_{PY} g')(PX,PZ)-\eta(PY)g'(PX,PZ)$$
$$-g'(T^{\bar \nabla}(PX,PY),PZ)$$
$$=(\nabla_{PX} g)(PY,PZ)+\eta(PX)g(PY,PZ)-(\nabla_{PY} g)(PX,PZ)-\eta(PY)g(PX,PZ)$$
$$-g(T^{\nabla}(PX,PY),PZ)=0$$
and the proof is complete.
\end{proof}

\begin{remark} We remark that if the screen distribution is not integrable, then the bracket is not well defined and, furthermore, the torsion on $S(TM')$ is not well defined. 
Regarding conditions on the integrability of a screen distribution, we refer to \cite{DS}.
\end{remark}

\begin{proposition}
Let $M'$ be a lightlike hypersurface of a pseudo-Riemannian manifold $(M,g)$, such that the screen distribution $S(TM')$ is integrable. If $(M,g,\eta,\nabla)$ and $(M,\widetilde g,\eta,\widetilde \nabla)$ are conformal-projective equivalent $SWMT$, then the induced structures on $S(TM')$, $(g',\eta',\bar \nabla)$ and $(\widetilde g',\eta',{\bar{\widetilde \nabla}})$, are conformal-projective equivalent, too.
\end{proposition}

\begin{proof}
Let $g'$ and $\widetilde g'$ be the induced metrics on $M'$. Obviously, the radical distributions with respect to $g'$ and $\widetilde g'$ are the same, since
\[
\widetilde g'=e^{\phi'+\psi'}g',
\]
where $\phi'$ and $\psi'$ denotes the restrictions of $\phi$ and $\psi$ to $M'$, respectively.

Let $X,Y\in  \Gamma^{\infty}(S(TM'))$. Then
\[
\bar{\widetilde \nabla}_XY+\widetilde\alpha(X,Y)\widetilde N=\widetilde \nabla_XY=\nabla_XY+d\phi(X)Y+d\phi(Y)X-g(X,Y)\nabla\psi
\]
\[
=\bar \nabla_XY+\alpha(X,Y)N+d\phi'(X)Y+d\phi'(Y)X-g'(X,Y)\bar \nabla\psi'-g'(X,Y) g(\nabla \psi,N)N,
\]
where $N$ is the defined before section of the lightlike transversal vector bundle of $M'$ with respect to $(S(TM'), g')$ and $\widetilde N=e^{-\frac{\phi +\psi}{2}}N$. By identifying the tangential components, we infer
\[
\bar {\widetilde \nabla}_XY=\bar \nabla_XY+d\phi'(X)Y+d\phi'(Y)X-g'(X,Y)\bar\nabla\psi',
\]
hence the conclusion.
\end{proof}

Let $M'$ be a lightlike hypersurface of a pseudo-Riemannian manifold $(M,g)$, such that the screen distribution $S(TM')$ is integrable.
Let $(g,\eta,\nabla)$ be a semi-Weyl structure with torsion on $M$ and let  $(g',\eta',\bar \nabla)$ be the induced semi-Weyl structure with torsion on $S(TM')$. Let
\[
\nabla_XY=\bar \nabla_XY+\alpha(X,Y)N, \ \ \nabla_XN=-B(X)+\tau(X)N,
\]
where $\alpha(X,Y):=g(\nabla_XY,N)$, $\tau(X):=g(\nabla_XN,N)$. Moreover, let $\nabla^*$ be the semi-dual connection of $\nabla$ with respect to $(g,\eta)$. Then
\[
\nabla^*_XY=(\bar \nabla)^*_XY+\alpha^*(X,Y)N, \ \ \nabla^*_XN=-B^*(X)+\tau^*(X)N,
\]
where $\alpha^*(X,Y):=g(\nabla^*_XY,N)$, $\tau^*(X):=g(\nabla^*_XN,N)$, and we define
\[ \ \beta(X,Y):=g(B(X),Y)=-g(\nabla_X N,Y),
\]
\[ \ \beta^*(X,Y):=g(B^*(X),Y)=-g(\nabla^*_X N,Y),
\]
for any $X,Y\in \Gamma^{\infty}(S(TM'))$.


\begin{proposition} Let $(M',g',\eta',\bar \nabla)$ be a lightlike hypersurface of a $SWMT$ $(M,g,\eta,\nabla)$, with the induced structure, such that the screen distribution $S(TM')$ is integrable. Then the $(0,2)$-tensor field $\beta$ (defined on $S(TM')$) is symmetric. Moreover, $( \alpha,\beta)=(\beta^*, \alpha^*)$.
\end{proposition}
\begin{proof} We have
\[
\beta(X,Y)-\beta(Y,X)=-g(\nabla_XN,Y)+g(\nabla_YN,X)
\]
\[
=-\Big(X(g(N,Y))-g(N,\nabla^*_XY)+\eta(X)g(N,Y)\Big)+\Big(Y(g(N,X))-g(N,\nabla^*_YX)+\eta(Y)g(N,X)\Big)
\]
\[=g(N,\nabla^*_XY)-g(N,\nabla^*_YX)=-g(N,T^{\nabla^*}(X,Y)+[X,Y])=-g(N,T^{\nabla^*}(X,Y)),
\]
for any $X,Y\in \Gamma^{\infty}(S(TM'))$, where $\nabla^*$ is the semi-dual connection of $\nabla$ with respect to $(g,\eta)$. In particular, the first statement follows from Proposition \ref{q4}.

Moreover
\[
\beta(X,Y)=g(N,\nabla^*_XY)=\alpha^*(X,Y),
\]
hence
\[
(\alpha,\beta)=(\beta^*,\alpha^*).
\]
\end{proof}

We pose the following.

\begin{definition} A point $x$ of a lightlike hypersurface $M'$, with an integrable screen distribution $S(TM')$, in a $SWMT$ $(M,g,\eta,\nabla)$,  is called a \emph{umbilical point} if there exists a real constant $c$ such that $\beta_x=cg'_x$. Moreover, $(M',S(TM'))$ is said to be a \emph{umbilical lightlike hypersurface} of $M$ if there exists a smooth function $f$ on $M'$ such that $\beta=fg'$.
\end{definition}
Now we can prove the following.

\begin{proposition}
The umbilical points of a lightlike hypersurface $M'$, with an integrable screen distribution $S(TM')$, in a $SWMT$, with the induced structure, are preserved by con\-for\-mal-projective changes.
\end{proposition}

\begin{proof}
Let $M'$ be a lightlike hypersurface, with an integrable screen distribution $S(TM')$, of a $SWMT$ $(M,g,\eta,\nabla)$ and let
$(g',\eta',\bar \nabla)$ be the induced structure on $S(TM')$.
Remark that the unit normal vector field $\widetilde N$ with respect to $\widetilde g$ equals
to $\widetilde N=e^{-\frac{\phi+\psi}{2}}N$.
Then, for any $X,Y\in \Gamma^{\infty}(TM)$, we have
\[
\widetilde \beta(X,Y)=-\widetilde g(\widetilde \nabla_X\widetilde N,Y)=-e^{\frac{\phi+\psi}{2}}g(\nabla_XN+d\phi(X)N+d\phi(N)X-g(X,N)\nabla\psi, Y)
\]
\[
-e^{\phi+\psi}X(e^{-\frac{\phi+\psi}{2}})g(N,Y)
\]
\[
=-e^{\frac{\phi+\psi}{2}}\Big(g(\nabla_XN, Y)+d\phi(N)g(X,Y)\Big)= e^{\frac{\phi+\psi}{2}}\Big(\beta(X,Y)-d\phi(N)g'(X,Y)\Big),
\]
hence
\[
\widetilde \beta=e^{\frac{\phi+\psi}{2}}(\beta-d\phi(N)g')
\]
and we immediately get the conclusions.
\end{proof}

As a consequence, we obtain
\begin{corollary}
Let $M'$ be a lightlike hypersurface of a pseudo-Riemannian manifold $(M,g)$, such that the screen distribution $S(TM')$ is integrable. If $(M,g,\eta,\nabla)$ and $(M,\widetilde g,\eta,\widetilde \nabla)$ are conformal-projective equivalent $SWMT$ and $(M',S(TM'))$ is a umbilical lightlike hypersurface of $(M,g,\eta,\nabla)$, then
 $(M',S(TM'))$ is a umbilical lightlike hypersurface of $(M,\widetilde g,\eta,\widetilde \nabla)$, too.
\end{corollary}

\section{Non-degenerate affine distributions}

We will prove that a $SWMT$ can be realized by a non-degenerate affine distribution.
Remark that Haba proved in \cite{ha} that a $SMT$ manifold can be realized by a non-degenerate equiaffine distribution.
We shall follow the natural idea for this construction.

Let $\{\omega,\xi\}$ be an affine distribution on the $n$-dimensional manifold $M$, i.e., $\omega$ is an $\mathbb{R}^{n+1}$-valued $1$-form on $M$ and $\xi$ is an $\mathbb{R}^{n+1}$-valued function on $M$, such that
$$\mathbb{R}^{n+1}=Im (\omega_x)\oplus \mathbb{R}\xi_x,$$
for any $x\in M$.

Moreover, we suppose $$Im (d\omega_x)\subset Im (\omega_x),$$ for any $x\in M$.

We define the $SWMT$
$(M, g,\eta,\nabla)$ in the following way. For any $X,Y\in \Gamma^{\infty}(TM)$, set
\[
X\omega(Y)=\omega(\nabla_XY)+g(X,Y)\xi,
\]
\[
X\xi=-\omega(B(X))+\eta(X)\xi,
\]
hence, $\nabla$ satisfies the properties of an affine connection, $g$ of a symmetric $(0,2)$-tensor field, $\eta$ of a $1$-form (and $B$ is a $(1,1)$-tensor field).

\begin{theorem}
If $\{\omega,\xi\}$ is a non-degenerate affine distribution on $M$, then $(M,g,\eta,\nabla)$ is a $SWMT$ and the curvature tensor of $\nabla$ has the following expression:
\[
R^\nabla (X,Y)Z=g(Y,Z)B(X)-g(X,Z)B(Y),
\]
for any $X,Y,Z\in \Gamma^{\infty}(TM)$.
\end{theorem}
\begin{proof}
We have:
\[
X(g(Y,Z)\xi)=X(g(Y,Z))\xi+g(Y,Z)(-\omega(B(X))+\eta(X)\xi)
\]
hence
\[
((\nabla_Xg)(Y,Z))\xi=X(g(Y,Z))\xi-g(\nabla_X Y,Z)\xi-g(Y,\nabla_X Z)\xi
\]
\[=X(g(Y,Z)\xi)-g(Y,Z)(-\omega(B(X))+\eta(X)\xi)-g(\nabla_X Y,Z)\xi-g(Y,\nabla_X Z)\xi
\]
\[=X(Y\omega(Z)-\omega(\nabla_YZ))+g(Y,Z)\omega(B(X))-g(Y,Z)\eta(X)\xi-g(\nabla_X Y,Z)\xi-g(Y,\nabla_X Z)\xi
\]
\[=XY\omega(Z)-X\omega(\nabla_Y Z)+g(Y,Z)\omega(B(X))-\big(g(Y,Z)\eta(X)+g(\nabla_X Y,Z)+g(Y,\nabla_X Z)\big)\xi
\]
and
\[
\big((\nabla_Xg)(Y,Z)-(\nabla_Yg)(X,Z)\big)\xi
\]
\[=\big(g(X,Z)\eta(Y)-g(Y,Z)\eta(X)-g(Z,\nabla_X Y)+g(Z,\nabla_Y X)+g([X,Y],Z)\big)\xi+
\]
\[+\omega(\nabla_Y \nabla_X Z-\nabla_X\nabla_Y Z-\nabla_{[X,Y]}Z+g(Y,Z)B(X)-g(X,Z)B(Y))
\]
thus
\[
(\nabla_Xg)(Y,Z)-(\nabla_Yg)(X,Z)+\eta(X)g(Y,Z)-\eta(Y)g(X,Z)+g(T^{\nabla}(X,Y),Z)=0
\]
and
\[R^\nabla (X,Y)Z=g(Y,Z)B(X)-g(X,Z)B(Y),
\]
for any $X,Y,Z\in \Gamma^{\infty}(TM)$. Then the proof is complete.
\end{proof}

Direct computations give the following.
\begin{corollary} If $\{\omega,\xi\}$ is a non-degenerate affine distribution on $M$, then the $SWMT$ $(M, g,\eta,\nabla)$ has the following Ricci and scalar curvature:
\[ \Ric(Y,Z)=g(Y,Z)\sum_{i=1}^n \epsilon_i g(B(E_i),E_i)-\sum_{i=1}^n \epsilon_i g(E_i,Z) g(B(Y),E_i),
\]
\[ \scal=(n-1)\sum_{i=1}^n\epsilon_i  g(B(E_i),E_i),
\]
for any $Y,Z\in \Gamma^{\infty}(TM)$, where $\{E_i\}_{1\leq i \leq n}$ is an orthonormal frame field with respect to the pseudo-Riemannian metric $g$, $g(E_i,E_j)=\epsilon_i \delta_{ij}$, $\epsilon_i=\pm 1$.
\end{corollary}

\begin{corollary}
If $\{\omega,\xi\}$ is a non-degenerate affine distribution on $M$, then the $SWMT$ $(M, g,\eta,\nabla)$ has symmetric Ricci tensor if and only if
$$\epsilon_i g(B(E_j),E_i)=\epsilon_j g(B(E_i),E_j),$$
for any $1\leq i,j\leq n$. Hence, if $B=cI$, where $c$ is a smooth function on $M$, then the Ricci tensor is symmetric. In this case, the scalar curvature is given by
\[\scal=c(n-1)(i_p-i_n),
\]
where $i_p$ (respectively, $i_n$) is the positivity index (respectively, the negativity index) of $g$. In particular, if $g$ is a neutral metric, that is, $i_p=i_n$, then the scalar curvature is zero.
\end{corollary}
Remark that, for a smooth function $\psi$ on $M$, the transformation $\xi\rightsquigarrow \widetilde \xi$, where
\[
\widetilde \xi:=e^{-\psi}\Big(\omega(\nabla\psi)+\xi\Big)
\]
gives rise to
\begin{equation} \label{e1}
\widetilde g=e^{\psi}g, \ \ \widetilde \eta=\eta, \ \ \widetilde \nabla=\nabla-g\otimes \nabla\psi.
\end{equation}
Moreover
\[
\widetilde B(X)=e^{-\psi}\Big(B(X)-\nabla_X\nabla \psi+X(\psi)\nabla \psi+\eta(X)\nabla\psi\Big).
\]

From Proposition \ref{p5z} (taking $\phi=0$), we can state
\begin{proposition}
If $\{\omega,\xi\}$ is a non-degenerate affine distribution on $M$, then $(M,\widetilde g,\eta,\widetilde \nabla)$ given by (\ref{e1}) is a $SWMT$.
\end{proposition}

Also, for a smooth function $\psi$ on $M$, the transformation $\xi\rightsquigarrow \widetilde \xi$, where
\[
\widetilde \xi:=\omega(\nabla\psi)+e^{-\psi}\xi
\]
gives rise to
\begin{equation}\label{e2}
\widetilde g=e^{\psi}g, \ \ \widetilde \eta=\eta+(e^{\psi}-1)d\psi, \ \ \widetilde \nabla=\nabla-e^{\psi}g\otimes \nabla\psi.
\end{equation}
Moreover
\[
\widetilde B(X)=e^{-\psi}B(X)-\nabla_X\nabla \psi+(e^{\psi}-1)X(\psi)\nabla \psi+\eta(X)\nabla\psi.
\]

By direct computations, we obtain
\begin{lemma} \label{lk}
\[
T^{\widetilde \nabla}=T^{\nabla},
\]
\[
(\widetilde \nabla_X\widetilde g)(Y,Z)-(\widetilde \nabla_Y\widetilde g)(X,Z)
\]
\[
=e^{\psi}\Big((\nabla_Xg)(Y,Z)+d\psi(X)g(Y,Z)-(\nabla_Yg)(X,Z)-d\psi(Y)g(X,Z)\Big),
\]
for any $X,Y,Z\in \Gamma^{\infty}(TM)$.
\end{lemma}

Thus we can state
\begin{proposition}
If $\{\omega,\xi\}$ is a non-degenerate affine distribution on $M$, then $(M,\widetilde g,\widetilde \eta-e^{\psi}d\psi,\widetilde \nabla)$ given by (\ref{e2}) is a $SWMT$.
\end{proposition}

\begin{proof}
For any $X,Y,Z\in \Gamma^{\infty}(TM)$, we get from Lemma \ref{lk}
\[
(\widetilde \nabla_X\widetilde g)(Y,Z)+\Big(\widetilde \eta(X)-e^{\psi}d\psi(X)\Big)\widetilde g(Y,Z)-(\widetilde \nabla_Y\widetilde g)(X,Z)-\Big(\widetilde \eta(Y)-e^{\psi}d\psi(Y)\Big)\widetilde g(X,Z)
\]
\[
=e^{\psi}\Big((\nabla_Xg)(Y,Z)+\eta(X)g(Y,Z)-(\nabla_Yg)(X,Z)-\eta(Y)g(X,Z) \Big)
\]
\[
=e^{\psi}g(T^{\nabla}(X,Y),Z)=\widetilde g(T^{\widetilde \nabla}(X,Y),Z),
\]
hence the conclusion.
\end{proof}


\section*{Declarations}

\textbf{Financial support.} Not applicable.
\\
\\
\textbf{Conflicts of interests.} The authors declare that there is no conflict of interests.

\noindent Adara M. BLAGA, \\
Department of Mathematics, \\
West University of Timi\c{s}oara, \\
Bld. V. P\^{a}rvan, 300223, Timi\c{s}oara, Rom\^{a}nia. \\
Email: adarablaga@yahoo.com

\bigskip

\noindent Antonella NANNICINI, \\
Department of Mathematics, \\
"U. Dini", University of Florence,\\
Viale Morgagni, 67/a, 50134, Firenze, Italy.\\
Email: antonella.nannicini@unifi.it

\end{document}